\documentclass[11pt]{amsart} 

\usepackage{amsmath, amssymb, amsfonts, tikz, hyperref, comment, url, geometry, todonotes}
\usetikzlibrary{math, trees}

\tikzset{black/.append style={circle,draw,fill=black,inner sep=1.5pt}, 
    white/.append style={circle,draw,fill=white,inner sep=1.5pt}}

\newtheorem{theorem}{Theorem}[section]
\newtheorem{lemma}[theorem]{Lemma}
\newtheorem{proposition}[theorem]{Proposition}
\newtheorem{corollary}[theorem]{Corollary}
\theoremstyle{definition}
\newtheorem{example}[theorem]{Example}
\newtheorem{remark}[theorem]{Remark}
\newtheorem{definition}[theorem]{Definition}
\newtheorem*{definition*}{Definition}
\newtheorem*{theorem*}{Theorem}
\newtheorem*{proposition*}{Proposition}
\newtheorem*{corollary*}{Corollary}

\newtheorem{conjecture}[theorem]{Conjecture}

\newcommand{\Z}{\mathbb{Z}}
\newcommand{\wt}{\mathrm{wt}}
\renewcommand{\O}{\mathcal{O}}

\author{Yibo Gao}
\address{Department of Mathematics, Massachusetts Institute of Technology, \mbox{Cambridge, MA 02139}}
\email{\href{mailto:gaoyibo@mit.edu}{{\tt gaoyibo@mit.edu}}}

\author{Andrew Gu}
\address{Department of Mathematics, Massachusetts Institute of Technology, \mbox{Cambridge, MA 02139}}
\email{\href{mailto:agu1@mit.edu}{{\tt agu1@mit.edu}}}

\begin{document}
\title{Arithmetic of weighted Catalan numbers}
\date{\today}

\begin{abstract}
In this paper, we study arithmetic properties of weighted Catalan numbers. Previously, Postnikov and Sagan found conditions under which the $2$-adic valuations of the weighted Catalan numbers are equal to the $2$-adic valutations of the Catalan numbers. We obtain the same result under weaker conditions by considering a map from a class of functions to $2$-adic integers. These methods are also extended to $q$-weighted Catalan numbers, strengthening a previous result by Konvalinka. Finally, we prove some results on the periodicity of weighted Catalan numbers modulo an integer and apply them to the specific case of the number of combinatorial types of Morse links. Many open questions are mentioned.
\end{abstract}
\maketitle

\section{Introduction}\label{sec:intro}
The sequence of Catalan numbers $\{C_n\}_{n\geq0}$ is one of the most well-studied sequences in combinatorics. They have a product formula $C_n=\binom{2n}{n}/(n+1)$ and count Dyck paths, binary trees, triangulations and many more classical combinatorial objects \cite{stanley2015catalan}. In this paper, we focus on weighted Catalan numbers, one of the many generalizations of Catalan numbers, and their arithmetic.

Recall that a \textit{Dyck path} of semilength $n$ is a sequence of points $\{(x_k,y_k)\}_{k=0}^{2n}$ that starts at $(x_0,y_0)=(0,0)$ and ends at $(x_{2n},y_{2n})=(2n,0)$  in the upper half-plane of the integer lattice $\Z^2$ such that each step $s_k=(x_{k}-x_{k-1},y_{k}-y_{k-1})$ is either $(1,-1)$ or $(1,1)$. For a fixed sequence of integers $b=(b(0),b(1),b(2),\ldots)$, define the \textit{weight} of a step $s_k$ to be $b(y_{k-1})$ if $s_k$ has the form $(1,1)$ and 1 if $s_k$ has the form $(1,-1)$. For a Dyck path $P$, define its \textit{weight} $\wt_b(P)$ to be the product of the weights of its steps. See Figure~\ref{fig:dyckpath} for an example.
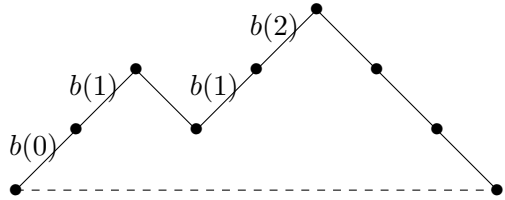
\begin{figure}[h!]
\centering
\begin{tikzpicture}[scale=0.8]
\node at (0,0) {$\bullet$};
\node at (1,1) {$\bullet$};
\node at (2,2) {$\bullet$};
\node at (3,1) {$\bullet$};
\node at (4,2) {$\bullet$};
\node at (5,3) {$\bullet$};
\node at (6,2) {$\bullet$};
\node at (7,1) {$\bullet$};
\node at (8,0) {$\bullet$};
\node at (0.3,0.7) {$b(0)$};
\node at (1.3,1.7) {$b(1)$};
\node at (3.3,1.7) {$b(1)$};
\node at (4.3,2.7) {$b(2)$};
\draw(0,0)--(2,2)--(3,1)--(5,3)--(8,0);
\draw[dashed](0,0)--(8,0);
\end{tikzpicture}
\caption{A Dyck path with weight $b(0)b(1)^2b(2)$.}
\label{fig:dyckpath}
\end{figure}
\begin{definition}
For $b:\Z_{\geq0}\rightarrow\Z$, the \textit{weighted Catalan numbers} $C_n^b$ are defined as
\[C_n^b=\sum_{P}\wt_b(P)\]
where the sum is over all Dyck paths $P$ of semilength $n$ and the weight of a Dyck path, $\wt_b(P)$, is explained as above.
\end{definition}
The weighted Catalan numbers have a beautiful generating function, illustrated in the following proposition, which is left as a simple exercise for the reader, whose proof can be found in the book of Goulden and Jackson \cite{goulden2004combinatorial}.
\begin{proposition}
\label{prop:genfunc}
The weighted Catalan numbers have the following generating function:
\[\sum_{n\geq0}C_n^bx^n=\cfrac{1}{1-\cfrac{b(0)x}{1-\cfrac{b(1)x}{1-\cfrac{b(2)x}{1-\cdots}}}}.\]
\end{proposition}
Specializations of weighted Catalan numbers count many interesting combinatorial objects. We provide some examples here. When $b\equiv1$, we recover the usual Catalan numbers. When $b(k)=k+1$, $C_n^b=(2n-1)!!$ counts the number of matchings of $2n$ objects. When $b(k)=(k+1)^2$, $C_n^b$ counts the number of alternating permutations of size $2n$ and when $b(k)=(k+1)(k+2)$, $C_n^b$ counts the number of alternating permutations of size $2n+1$ \cite{goulden2004combinatorial}. Finally, when $b(k)=(2k+1)^2$, $C_n^b$ counts the number of combinatorial types Morse links of order $n$ \cite{postnikov2000morse}.

The divisibility of Catalan numbers by powers of 2 has been determined both arithmetically and combinatorially (see for example \cite{dickson1966history} and \cite{deutsch2006congruences}). For a positive integer $q\geq2$ and for $n\in\Z_{>0}$, let $\xi_q(n)$ be the unique $m\in\Z_{\geq0}$ such that $q^m\mid n$ and $q^{m+1}\nmid n$ and let $s_q(n)$ be the sum of digits in the $q$-ary expansion of $n$. It is then well-known that $\xi_2(C_n)=s_2(n+1)-1$.

Postnikov and Sagan \cite{postnikov2007power} provided a sufficient condition on $b:\Z_{\geq0}\rightarrow\Z$ for $\xi_2(C_n^b)=\xi_2(C_n)$. For any function $f:\Z_{\geq0}\rightarrow\Z$, let $\Delta f:\Z_{\geq0}\rightarrow\Z$ be defined as $(\Delta f)(x)=f(x+1)-f(x)$.
\begin{theorem}[\cite{postnikov2007power}]\label{thm:postsagan}
If $b:\Z_{\geq0}\rightarrow\Z$ satisfies
\begin{enumerate}
    \item $b(0)$ is odd,
    \item $2^{n+1}\mid(\Delta^nb)(x)$ for all $x\in\Z_{\geq0}$,
\end{enumerate}
then $\xi_2(C_n^b)=\xi_2(C_n)=s_2(n+1)-1$ for all $n$.
\end{theorem}
In the case where $b$ is a polynomial function, a necessary and sufficient condition for $\xi_2(C_n^b)=\xi_2(C_n)$ for all $n$ was conjectured by Konvalinka \cite{konvalinka2007divisibility} and resolved by An \cite{an2010combinatorial} arithmetically. A generalization of Theorem~\ref{thm:postsagan} to weighted $q$-Catalan numbers, for a prime power $q$, is given by Konvalinka \cite{konvalinka2007divisibility}.

In this paper, we relax conditions of Theorem~\ref{thm:postsagan} and also arrive at the conclusion that $\xi_2(C_n^b)=\xi_2(C_n)$. Our main theorem is stated and proved in Section~\ref{sec:thm} using combinatorial arguments on binary trees. In Section~\ref{sec:generalization}, we show that a similar approach can be used to generalize our results to weighted $q$-Catalan numbers. In addition, in a different setting, we discuss modulo arithmetic of weighted Catalan numbers in Section~\ref{sec:period}, resolving some periodicity conjectures by Postnikov \cite{postnikov2000morse}. Many intriguing conjectures will be mentioned along the way.

\section{The main theorem}\label{sec:thm}

\begin{theorem}\label{thm:main}
If $b: \Z_{\geq 0}\to\Z$ satisfies 
\begin{enumerate}
    \item $b(0)$ is odd,
    \item $4\mid (\Delta b)(x)$ for all $x\in\mathbb{Z}_{\geq 0}$,
    \item $2^n\mid (\Delta^nb)(x)$ for all $n\geq 2$ and $x\in \mathbb{Z}_{\geq 0}$,
\end{enumerate}
then
$\xi_2(C_n^b)=\xi_2(C_n)=s_2(n+1)-1.$
\end{theorem}

To prove the main theorem, we modify the approach of Postnikov and Sagan with average weight functions $r_b(\mathcal{O}; x)$ attached to orbits of binary trees. Under their conditions, these functions were always odd-valued, whereas they are only integer-valued under the weaker conditions of Theorem \ref{thm:main}. We develop a method of analyzing their parity and show via induction that we have an odd number of minimal orbits with odd-valued average weight functions, which will prove the theorem.

\subsection{Binary trees and minimal orbits} \label{subsec:binarytrees}

In this section we reprove some important facts about binary trees.

\begin{definition}
A \emph{binary tree} is a rooted tree where each vertex has a left child, a right child, both children, or no children.
\end{definition}

Let $\mathcal{T}_n$ denote the set of binary trees on $n$ vertices. One of the interpretations of Catalan numbers is that $C_n=\lvert \mathcal{T}_n\rvert$. There is a symmetry group $G_n$ acting on $\mathcal{T}_n$, generated by reflections about a vertex which switch the left and right subtrees. An orbit of trees refers to an orbit under the action of $G_n$, which groups together binary trees which are isomorphic when the distinction between left and right is disregarded.

The following lemma was shown by Deutsch and Sagan \cite{deutsch2006congruences}.
\begin{lemma}
\label{lem:orbit sizes}
Let $\mathcal{O}$ be an orbit of the symmetry group $G_n$ on $\mathcal{T}_n$. Then $\lvert \mathcal{O}\rvert$ is of the form $2^t$ for a nonnegative integer $t$. Furthermore, $t\geq s=s_2(n+1)-1$, with equality for $(2s-1)!!$ orbits.
\end{lemma}

We will use \emph{minimal orbit} to refer to an orbit of the minimum possible size, which is $2^s$. The set of all orbits of $\mathcal{T}_n$ under $G_n$ is denoted $\mathcal{U}_n$, and the set of minimal orbits is denoted $\mathcal{U}^{\text{min}}_n$.

\begin{definition}
The \emph{complete binary tree of depth $k$} is the binary tree of $2^k-1$ vertices where each vertex in layer $0, 1, \dotsc, k-2$ has two children and each vertex in layer $k-1$ has no children.
\end{definition}
Note that complete binary trees are exactly the trees whose orbit consist of just one tree.

Now we have the following key result which describes the structure of minimal orbits of trees on $n$ vertices. (Recall that these are orbits of size $2^{s_2(n+1)-1}$.) This was done more generally for $q$-ary trees by Konvalinka \cite{konvalinka2007divisibility}.

\begin{theorem}[Structure of minimal orbits]
\label{thm:orbit structure}
Let $s=s(n+1)-1$ and $n+1=2^{k_1}+\dotsb+2^{k_{s+1}}$ be the binary expansion of $n+1$. Then all trees in an orbit of $\mathcal{U}_n^{\text{min}}$ may be constructed in the following manner: construct an arbitrary binary tree with $s$ vertices, and then attach completely symmetric trees of depth $k_1, \dotsc, k_{s+1}$ to the $s+1$ endpoints.
\end{theorem}

See Figure \ref{fig:minorbit} for an example. In constructing these minimal orbits, the $s$ vertices of the ``arbitrary binary tree'' will be marked black while the other vertices will be marked white.

\begin{figure}[h!]
\centering
    \begin{tikzpicture}[
    level distance=1cm,
    level 1/.style={sibling distance=4cm},
    level 2/.style={sibling distance=3cm},
    level 3/.style={sibling distance=1.5cm},
    level 4/.style={sibling distance=0.75cm},
    level 5/.style={sibling distance=0.375cm},
    level 6/.style={sibling distance=0.1875cm}]
      \node[black] {}
        child {node[black] {}
          child {node[white] {}
            child {node[white] {}
            child {node[white] {}
            child {node[white] {}
            child {node[white] {}}
            child {node[white] {}}}
            child {node[white] {}
            child {node[white] {}}
            child {node[white] {}}}
            }
            child {node[white] {}
            child {node[white] {}
            child {node[white] {}}
            child {node[white] {}}}
            child {node[white] {}
            child {node[white] {}}
            child {node[white] {}}}
            }
            }
            child {node[white] {}
            child {node[white] {}
            child {node[white] {}
            child {node[white] {}}
            child {node[white] {}}}
            child {node[white] {}
            child {node[white] {}}
            child {node[white] {}}}
            }
            child {node[white] {}
            child {node[white] {}
            child {node[white] {}}
            child {node[white] {}}}
            child {node[white] {}
            child {node[white] {}}
            child {node[white] {}}}
            }
            }
            }
          child {node[white] {}
          child {node[white] {}}
          child {node[white] {}}}
        }
        child {node[black] {}
        child {node[black, xshift=1.5cm] {}
        child {node[white] {}}
        child {node[white] {}
            child {node[white] {}
            child {node[white] {}
            child {node[white] {}}
            child {node[white] {}}}
            child {node[white] {}
            child {node[white] {}}
            child {node[white] {}}}
            }
            child {node[white] {}
            child {node[white] {}
            child {node[white] {}}
            child {node[white] {}}}
            child {node[white] {}
            child {node[white] {}}
            child {node[white] {}}}
            }
            }}
        };
    \end{tikzpicture}
\caption{An example for $n=54, n+1=2^5+2^4+2^2+2^1+2^0$.}
\label{fig:minorbit}
\end{figure}
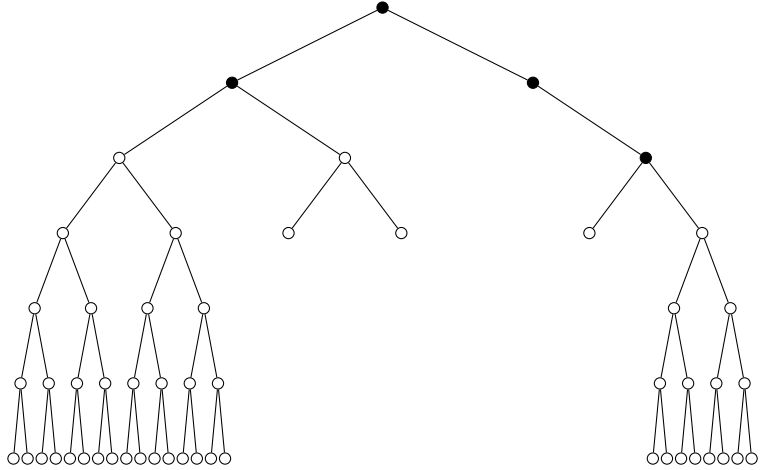
Furthermore, this mapping can be reversed in the following way: for every vertex, color it white if the left and right subtrees at that vertex are isomorphic, and otherwise color it black. The black vertices are the original $s$ vertices and all white vertices are part of complete binary trees.

The symmetries of such trees are generated precisely by reflections about the original $s$ vertices (the ones marked in black).

\begin{remark}
In Theorem \ref{thm:orbit structure}, the term $2^0$ may appear in the binary expansion of $n+1$. This corresponds to an empty tree. In the example of Figure \ref{fig:minorbit}, all the black vertices have two children except for the third one from the left, which has only a right child. This is because the empty tree was assigned to its left child.
\end{remark}
\subsection{Mapping $\mathcal{F}$ to $2$-adic integers} \label{subsec:2-adics}

Let $\mathcal{F}$ denote the set of functions $f: \mathbb{Z}_{\geq 0}\to\mathbb{Z}$ such that $2^n\mid (\Delta^nf)(x)$ for all $n\geq 0$ and $x\in \mathbb{Z}_{\geq 0}$. Define the shift operator $S$ by
\[(Sf)(x)=f(x+1).\]

\begin{lemma}
\label{lem:productrule}
The product rule
\[\Delta^n(f\cdot g)=\sum_{k=0}^{n}\binom{n}{k}\Delta^{n-k}\left(S^k(f)\right)\Delta^k(g)\]
holds, which can be extended to multiple functions as
\[\Delta^n(f_1\dotsm f_m)=\sum_{a_1+\dotsb+a_m=n}\binom{n}{a_1, \dotsc, a_m}\Delta^{a_1}(S^{a_2+\dotsb+a_m}f_1)\cdot \Delta^{a_2}(S^{a_3+\dotsb+a_m}f_2)\dotsm \Delta^{a_m}(f_m).\]
\end{lemma}

\begin{lemma}
\label{lem:F closure}
The set $\mathcal{F}$ is closed under the following operations:
\[f\mapsto Sf, \quad (f, g)\mapsto f\cdot g, \quad \text{ and }\quad (f, g)\mapsto \langle f, g\rangle = \frac{f(x+1)g(x)+f(x)g(x+1)}{2}.\]
\end{lemma}
\begin{proof}
    Closure under $S$ is clear. For multiplication, we may use the product rule: each individual term in the expansion of $\Delta^n(f\cdot g)$ will be divisible by $2^n$. Likewise, by writing
    \[\langle f, g\rangle = f\cdot g + \frac{\Delta(f)\cdot g+f\cdot \Delta(g)}{2}\]
    and expanding with the product rule, we can show that $\langle f, g\rangle\in\mathcal{F}$ if $f, g\in\mathcal{F}$.
\end{proof}

Next, observe that for $f\in \mathcal{F}$ and any nonnegative integer $n$, we have $(\Delta^nf)(x)\equiv 0\pmod{2^n}$ by definition and also $(\Delta^{n+1}f)(x)\equiv 0\pmod{2^{n+1}}$. Then $\Delta^nf$ is constant modulo $2^{n+1}$ and must be equal to either $2^{n}$ or $0$. Define a mapping
\[\varepsilon: \mathcal{F}\to\{0, 1\}^{\mathbb{N}}\]
by sending $f$ to $(\varepsilon_0^f, \varepsilon_1^f, \dotsc)$ where 
\[
\varepsilon_n^f=\begin{cases}
0 & \text{ if }\Delta^nf\equiv 0\pmod {2^{n+1}} \\
1 & \text{ if }\Delta^nf\equiv 2^n\pmod{2^{n+1}}
\end{cases}.
\]

Define the weight of a binary tree $T$ by
\[w_b(T; x)=\prod_{v\in T}b(x+l_v)\]
where the product is over all vertices $v$ of $T$ and $l_v$ the number of left-edges in the path from the root to $v$. Evaluating this weight at $x=0$ and summing over all binary trees gives $C^b_n$. For each orbit $\mathcal{O}$, define a total weight function $w_{\mathcal{O}}(x)$ by
\[w_{\mathcal{O}}(x)=\sum_{T\in\mathcal{O}}w_b(T)\]
and the average weight function $r_b(\mathcal{O}; x)$ by
\[r_b(\mathcal{O}; x)=\frac{w_b(\mathcal{O}; x)}{\lvert\mathcal{O}\rvert}.\]
We have the recursive formula
\[r_b(\mathcal{O}; x)=b(x)\cdot\langle r_b(\mathcal{O}_L; x), r_b(\mathcal{O}_R; x)\rangle\]
(see Postnikov and Sagan \cite{postnikov2007power}) where $\mathcal{O}_L$ and $\mathcal{O}_R$ are the orbits associated to the left and right subtrees of the root. By Lemma \ref{lem:F closure} and induction, all functions $r_b(\mathcal{O}; x)$ are in $\mathcal{F}$.

\begin{definition}\label{def:epsilonO}
For an orbit $\mathcal{O}$ of binary trees, let $\varepsilon^{\mathcal{O}}_m=\varepsilon^f_m$ for $f=r_b(\O)\in \mathcal{F}$.
\end{definition}

From now on, we will also let $\varepsilon_k=\varepsilon^b_k$ for all $k$, i.e. the superscript $b$ will be omitted.

\begin{lemma}
\label{lem:epsilon formula}
For an orbit $\mathcal{O}$ with left and right orbits $\mathcal{O}_L, \mathcal{O}_R$, the following formula holds:
\[\varepsilon^{\mathcal{O}}_{m} \equiv\sum_{i+j+k=m}\binom{m}{i, j, k}\varepsilon_{k}(\varepsilon^{\mathcal{O}_L}_{i}\varepsilon^{\mathcal{O}_R}_j+\varepsilon^{\mathcal{O}_L}_{i+1}\varepsilon^{\mathcal{O}_R}_j+\varepsilon^{\mathcal{O}_L}_{i}\varepsilon^{\mathcal{O}_R}_{j+1})\pmod{2}.\]
\end{lemma}
\begin{proof}
    We make use of the identity 
    \[\langle f, g\rangle = f\cdot g + \frac{\Delta(f)\cdot g+f\cdot \Delta(g)}{2}.\]
    By expanding with the product rule, we have
    \begin{align*}
        \Delta^mr_b(\mathcal{O}; x)=\sum_{i+j+k=m} &\binom{m}{i, j, k}\Big[\Delta^k(S^{i+j}b)\cdot\Big(\Delta^{i}(S^jr_b(\mathcal{O}_L; x))\cdot \Delta^j(r_b(\mathcal{O}_R; x))+ \\
        &\frac{\Delta^{i+1}(S^jr_b(\mathcal{O}_L; x))\cdot \Delta^j(r_b(\mathcal{O}_R; x))+\Delta^{i}(S^jr_b(\mathcal{O}_L; x))\cdot \Delta^{j+1}(r_b(\mathcal{O}_R; x))}{2}\Big)\Big].
    \end{align*}
    We may now break this up into three sums, each of which will correspond to a term in the statement in the lemma. The first term is
    \[\sum_{i+j+k=m}\binom{m}{i, j, k}\Delta^k(S^{i+j}b)\cdot\Delta^{i}(S^jr_b(\mathcal{O}_L; x))\cdot \Delta^j(r_b(\mathcal{O}_R; x))\]
    As $b, r_b(\mathcal{O}_L; x), r_b(\mathcal{O}_R; x)\in \mathcal{F}$, each term in this summation is divisible by $2^k\cdot 2^i\cdot 2^j=2^m$. This term is \emph{not} divisible by $2^{m+1}$ if and only if all of the following hold:
    \begin{itemize}
        \item $\binom{m}{i, j, k}$ is odd.
        \item $\Delta^k(S^{i+j}b)\equiv 2^k\pmod{2^{k+1}}$.
        \item $\Delta^i(S^jr_b(\mathcal{O}_L; x))\equiv 2^i\pmod{2^{i+1}}$.
        \item $\Delta^{j}(r_b(\mathcal{O}_R; x))\equiv 2^j\pmod{2^{j+1}}$.
    \end{itemize}
    Therefore the residue modulo $2^{m+1}$ is equal to the residue of
    \[\sum_{i+j+k=m}\binom{m}{i, j, k}\varepsilon_k\varepsilon^{\mathcal{O}_L}_i\varepsilon^{\mathcal{O}_R}_{j}\]
    modulo $2$, multiplied by $2^m$.
    
    Similarly, the second term in the expansion is
    \[\frac{1}{2}\sum_{i+j+k=m}\binom{m}{i, j, k}\Delta^k(S^{i+j}b)\cdot\Delta^{i+1}(S^jr_b(\mathcal{O}_L; x))\cdot \Delta^j(r_b(\mathcal{O}_R; x)).\]
    The summation without the factor of $\frac{1}{2}$ is always a multiple of $2^{m+1}$, and will not be a multiple of $2^{m+2}$ if and only if $\binom{m}{i, j, k}$ is odd and $\varepsilon^b_k=\varepsilon^{\mathcal{O}_L}_{i+1}=\varepsilon^{\mathcal{O}_R}_{j}=1$. The third term can be handled similarly. Summing all the equations yields the lemma.
    
    In the case where $\mathcal{O}_L$ or $\mathcal{O}_R$ is empty, the associated function is $r_b=1$ with $\varepsilon^{r_b}_0=1$ and $\varepsilon^{r_b}_k=0$ for $k\geq 1$. The proof of the lemma works the same in this case.
\end{proof}

\begin{lemma}
    \label{lem:reduction}
    Let $\mathcal{O}$ be an orbit of binary trees on $n$ vertices, and suppose that for some vertex $v$, the subtree with root $v$ is a complete binary tree of depth $k$. Let $\mathcal{O}'$ be the same orbit of trees with the subtree at $v$ replaced by a single vertex at $v$. Then \[\varepsilon_m^{\mathcal{O}}=\varepsilon_0^{2^k-2}\varepsilon_m^{\mathcal{O}'}.\]
\end{lemma}

\begin{proof}
    We prove this lemma by induction. If $v$ is not the root vertex, then we use lemma \ref{lem:epsilon formula} to expand. Exactly one of the left and right subtrees contains $v$, so the inductive hypothesis applies to factor out $\varepsilon_0^{2^k-2}$. Hence we can reduce to the base case where $v$ is the root, i.e. we have a complete binary tree of depth $k$. To prove this case, we induct on $k$.
    
    For $k=1$ the lemma follows because $r_b(\mathcal{O}; x)=b(x)$. For the inductive step, we use Lemma \ref{lem:epsilon formula} where $\mathcal{O}_L$ and $\mathcal{O}_R$ are complete binary trees of depth $k-1$. Note that by symmetry of $i$ and $j$, the $\varepsilon^{\mathcal{O}_L}_{i+1}\varepsilon^{\mathcal{O}_R}_j+\varepsilon^{\mathcal{O}_L}_{i}\varepsilon^{\mathcal{O}_R}_{j+1}$ term cancels out in the sum, and $\varepsilon^{\mathcal{O}_L}_{i}\varepsilon^{\mathcal{O}_R}_j$ also cancels out for all terms where $i\neq j$. Finally, note that if $i=j > 0$, then the binomial coefficient $\binom{m}{i, j, k}=\binom{m}{k}\binom{i+j}{i}$ is divisible by $\binom{2i}{i}$, which is even. Therefore the only remaining term in the summation is the $i=j=0$ term. By the inductive hypothesis, we have
    \[\varepsilon^{\mathcal{O}}_m\equiv \varepsilon_m^{\mathcal{O'}}\varepsilon_0^{2(2^{k-1}-1)}=\varepsilon_0^{2^k-2}\varepsilon_m^{\mathcal{O}'}\]
    as desired.
\end{proof}

The recursive formula illustrated in Lemma~\ref{lem:epsilon formula} enables us to explicitly write down $\varepsilon_m^{\O}$ in an elegant way. This combinatorial description gives another quick proof to Lemma~\ref{lem:reduction} but is not used in other places of the paper. Readers are free to skip to the proof of the main theorem.

We view an orbit $\O$ as a rooted binary tree. Traditionally, we say that a vertex $v$ is a descendant of another vertex $w$ in $\O$ if the unique path connecting $v$ and the root passes through $w$. Similarly, we say that a vertex $v$ is a \textit{descendant} of an edge $e$ in $\O$ if the unique path connecting $v$ and the root contains $e$. In addition, we say that an edge $e$ of $\O$ connecting two vertices $u$ and $v$ \textit{originates from} $u$ if $v$ is a descendant of $e$ (or of $u$). Consequently, two edges $e$ and $e'$ are called \textit{siblings} if they originate from the same vertex.
\begin{definition}\label{def:coin}
A \textit{coin-configuration} $C$ on $\O$ of order $m$ consists of the following data:
\begin{itemize}
    \item a subset of edges $E$ of $\O$ such that no siblings are simultaneously selected,
    \item a distribution of distinct \textit{coins} $\{1,\ldots,m\}\cup E$ to all vertices of $\O$ such that each coin $c_e$ labeled by an edge $e\in E$ is assigned to a descendent of $e$.
\end{itemize}
For such a coin-configuration $C$, let $C_v$ denote the set of coins at $v$. Define its weight $\wt(C)=\prod_{v\in\O}\varepsilon_{|C_v|}$ where the product is over all vertices $v$ in $\O$. Moreover, let $\mathcal{C}^{\O}_m$ denote the set of all coin-configurations on $\O$ of order $m$. 
\end{definition}
See Figure~\ref{fig:coin} for an example of a coin-configuration of order 9 on an orbit $\O$ which is a rooted binary tree on 8 vertices, where the selected edges in $E$ are highlighted as thick line segments and the distribution of coins is written underneath the vertices. This particular example has weight $(\varepsilon_0)^2(\varepsilon_1)(\varepsilon_2)^4(\varepsilon_3)$.
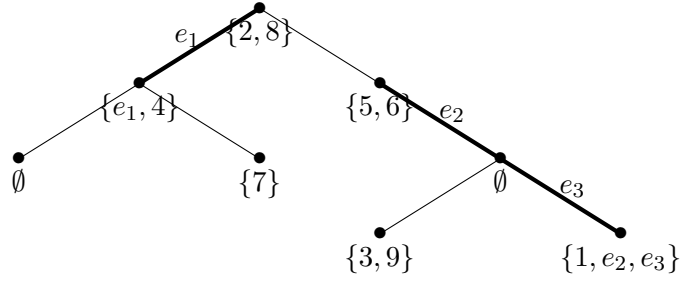
\begin{figure}[h!]
\centering
\begin{tikzpicture}[x=1.6cm]
\node at (0,0) {$\bullet$};
\node[below] at (0,0) {$\{2,8\}$};
\node at (-1,-1) {$\bullet$};
\node[below] at (-1,-1) {$\{e_1,4\}$};
\node at (-2,-2) {$\bullet$};
\node[below] at (-2,-2) {$\emptyset$};
\node at (0,-2) {$\bullet$};
\node[below] at (0,-2) {$\{7\}$};
\node at (1,-1) {$\bullet$};
\node[below] at (1,-1) {$\{5,6\}$};
\node at (2,-2) {$\bullet$};
\node[below] at (2,-2) {$\emptyset$};
\node at (3,-3) {$\bullet$};
\node[below] at (3,-3) {$\{1,e_2,e_3\}$};
\node at (1,-3) {$\bullet$};
\node[below] at (1,-3) {$\{3,9\}$};
\draw (-2,-2)--(-1,-1)--(0,0)--(1,-1)--(2,-2)--(3,-3);
\draw (0,-2)--(-1,-1);
\draw (2,-2)--(1,-3);
\draw[ultra thick] (0,0)--(-1,-1);
\draw[ultra thick] (1,-1)--(2,-2);
\draw[ultra thick] (2,-2)--(3,-3);
\node at (-0.6,-0.4) {$e_1$};
\node at (1.6,-1.4) {$e_2$};
\node at (2.6,-2.4) {$e_3$};
\end{tikzpicture}
\caption{An example of a coin-configuration of order 9 with weight $(\varepsilon_0)^2(\varepsilon_1)(\varepsilon_2)^4(\varepsilon_3)$.}
\label{fig:coin}
\end{figure}

\begin{proposition}\label{prop:coin}
We have the following explicit formula for an orbit $\O$:
$$\varepsilon_m^{\O}=\sum_{C\in\mathcal{C}^{\O}_m}\wt(C)$$
where the sum is over all coin-configurations $C$ on $\O$ of order $m$.
\end{proposition}
\begin{proof}
The proposition is a direct consequence of Lemma~\ref{lem:epsilon formula} and induction on the number of vertices of $\O$. For the base case where $\O$ is a single vertex, by definition, $\varepsilon^{\O}_m$ equals $\varepsilon_m$. At the same time, we have exactly 1 coin-configuration with weight $\varepsilon_m$ since there are no edges and all $m$ coins must be assigned to the unique vertex in $\O$.

Now assume that our proposition is true for all orbits with less than $n$ vertices and consider an orbit $\O$ with $n$ vertices. Similarly as in Lemma~\ref{lem:epsilon formula}, let $\O_L$ and $\O_R$ be the left and right subtrees of the root respectively and let $e_L$ and $e_R$ be the left and right edges originating from the root respectively. Consider $\sum_{C\in\mathcal{C}^{\O}_m}\wt(C)$, which is the RHS of the proposition. For a coin-configuration $C$ with a chosen subset $E$ of the edges, we cannot have both $e_L$ and $e_R$ lie in $E$. Then there are 3 possibilities: (1) $e_L\in E$, $e_R\notin E$; (2) $e_L\notin E$, $e_R\in E$; (3) $e_L,e_R\notin E$. For (1), we have $\binom{m}{i,j,k}$ ways to assign $k$ coins to the root, $i+1$ total coins to the left subtree ($i$ original coins in $\{1,\ldots,m\}$ and a new coin from $e_L$), and $j$ total coins to the right subtree. This provides a sum of $$\sum_{i+j+k=m}\binom{m} {i,j,k}\varepsilon_k\sum_{C\in\mathcal{C}^{\O_L}_{i+1}}\wt(C)\sum_{C\in\mathcal{C}^{\O_R}_{j}}\wt(C)=\sum_{i+j+k=m}\binom{m} {i,j,k}\varepsilon_k\varepsilon_{i+1}^{\O_L}\varepsilon_{j}^{\O_R}$$
by induction. Similarly handle case (2) and (3) and take the total sum $\sum_{C\in\mathcal{C}^{\O}_m}\wt(C)$, which is the exact same expression as RHS of Lemma~\ref{lem:epsilon formula} so it equals $\varepsilon_m^{\O}$ as desired.
\end{proof}

From the description of $\varepsilon_m^{\O}$ using coin-configurations (Definition~\ref{def:coin}), Lemma~\ref{lem:reduction} becomes immediate: if a vertex $v$ of $\O$ has isomorphic left and right subtrees, we can pair up coin-configurations with nontrivial coin assignments to descendents of $v$ by switching the left and right subtrees of $v$ and each pair of coin-configurations have the same weight.
\subsection{Proof of the Main Theorem}\label{subsec:proof}

Now we are ready to prove Theorem \ref{thm:main}.

\begin{proof}[Proof of Theorem \ref{thm:main}]
    We write
    \[C^b_n=\sum_{\mathcal{O}\in\mathcal{U}_n}\lvert \mathcal{O}\rvert\cdot r_b(\mathcal{O}; 0).\]
    From Lemma \ref{lem:orbit sizes}, $\lvert \mathcal{O}\rvert$ is always a power of two with exponent at least $s=s_2(n+1)-1$. We also know that $r_b(\mathcal{O}; 0)$ is always an integer and by definition, $r_b(\mathcal{O}; 0)$ is odd if and only if $\varepsilon^{\mathcal{O}}_0=1$. Then
    \[C^b_n\equiv \Big(\sum_{\mathcal{O}\in\mathcal{U}_n^{\text{min}}}\varepsilon^{\O}_0\Big)2^s\pmod{2^{s+1}},\]
    so it is equivalent to show that $\sum_{\mathcal{O}\in\mathcal{U}_n^{\text{min}}}\varepsilon^{\O}_0$ is odd.

    We use Theorem \ref{thm:orbit structure} to construct all minimal orbits of trees on $n$ vertices. By Corollary \ref{lem:reduction}, if we replace each complete binary tree of depth greater than $0$ by a single vertex, we get a new orbit of binary trees (on fewer vertices) which has the same sequence $\varepsilon^{\mathcal{O}}_m$, only reduced by a power of $\varepsilon_0$. This operation will be referred to as ``reduction.'' In Figure \ref{fig:oddcase}, have reduced the tree from Figure \ref{fig:minorbit} and obtained a new tree. Note that we have preserved the colors of the vertices, as we will need to refer to them later.
    
    \begin{figure}[h!]
    \centering
    \begin{tikzpicture}[
    level distance=1cm,
    level 1/.style={sibling distance=3cm},
    level 2/.style={sibling distance=2cm},
    level 3/.style={sibling distance=1.5cm},
    level 4/.style={sibling distance=0.75cm},
    level 5/.style={sibling distance=0.375cm},
    level 6/.style={sibling distance=0.1875cm}]
      \node[black] {}
        child {node[black] {}
          child {node[white] {}
            child {node[white] {}
            child {node[white] {}
            child {node[white] {}
            child {node[white] {}}
            child {node[white] {}}}
            child {node[white] {}
            child {node[white] {}}
            child {node[white] {}}}
            }
            child {node[white] {}
            child {node[white] {}
            child {node[white] {}}
            child {node[white] {}}}
            child {node[white] {}
            child {node[white] {}}
            child {node[white] {}}}
            }
            }
            child {node[white] {}
            child {node[white] {}
            child {node[white] {}
            child {node[white] {}}
            child {node[white] {}}}
            child {node[white] {}
            child {node[white] {}}
            child {node[white] {}}}
            }
            child {node[white] {}
            child {node[white] {}
            child {node[white] {}}
            child {node[white] {}}}
            child {node[white] {}
            child {node[white] {}}
            child {node[white] {}}}
            }
            }
            }
          child {node[white] {}
          child {node[white] {}}
          child {node[white] {}}}
        }
        child {node[black] {}
        child {node[black, xshift=1cm] {}
        child {node[white] {}}
        child {node[white] {}
            child {node[white] {}
            child {node[white] {}
            child {node[white] {}}
            child {node[white] {}}}
            child {node[white] {}
            child {node[white] {}}
            child {node[white] {}}}
            }
            child {node[white] {}
            child {node[white] {}
            child {node[white] {}}
            child {node[white] {}}}
            child {node[white] {}
            child {node[white] {}}
            child {node[white] {}}}
            }
            }}
        };
        
    \node[black, xshift=7cm] {}
        child {node[black] {}
          child {node[white] {}}
          child {node[white] {}
          }
        }
        child {node[black] {}
        child {node[black, xshift=1cm] {}
        child {node[white] {}}
        child {node[white] {}
            }
            }
        };
    \end{tikzpicture}
    \caption{The example tree from Figure \ref{fig:minorbit}, and its reduction on the right.}
    \label{fig:oddcase}
    \end{figure}
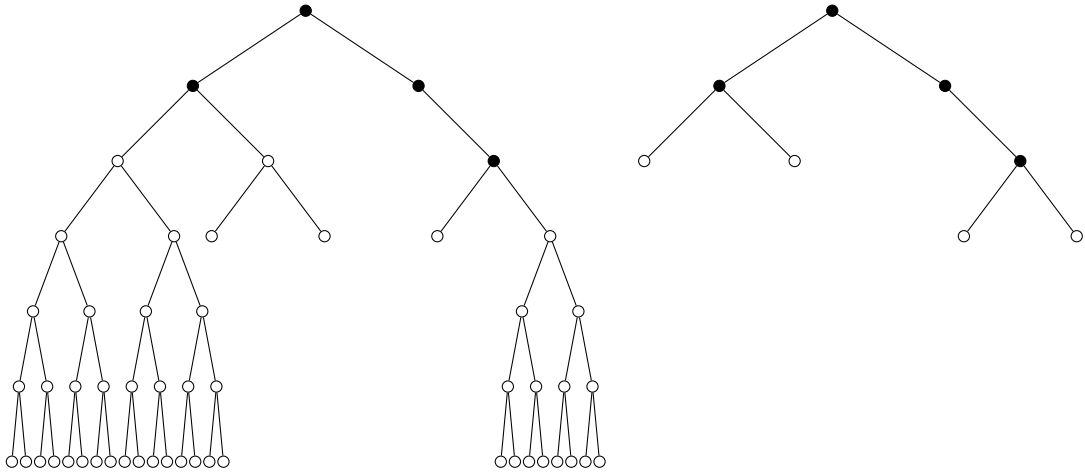
    
    We first consider the number of vertices remaining after reduction. If $n+1$ is odd, then there are $2s$ vertices remaining because there are $s$ nonempty complete binary trees being attached to $s$ vertices. If $n+1$ is even, then the number of remaining vertices is $2s+1$. In either case, let $k$ be the number of vertices of the trees of the minimal orbits after the reduction. With $\mathcal{U}_k$ as the set of orbits on $k$ vertices, let $f(\mathcal{O})$ denote the number of times an orbit $\mathcal{O}$ appears after reduction. Then we have
    \[\sum_{\mathcal{O}\in\mathcal{U}_n^{\text{min}}}\varepsilon^{\mathcal{O}}_0 = \varepsilon_0^{n-k}\sum_{\mathcal{O}\in \mathcal{U}_k}f(\mathcal{O})\varepsilon^{\mathcal{O}}_0.\]
    We want to determine the left hand side modulo $2$, so we can ignore all orbits for which $f(\mathcal{O})$ is even.
    
    In fact, we claim that those trees which remain (appear an odd number of times) are those representing minimal orbits of trees on $k$ vertices. Recall with Theorem \ref{thm:orbit structure} we assigned certain vertices as white and others as black, and we maintain this color scheme. After the reduction, we always have $s$ or $s+1$ white leaves attached to $s$ black vertices. See Table \ref{tab:reducedtrees} for an example, with all minimal orbits of trees on $14$ vertices shown with their reductions. In Table \ref{tab:reducedcount}, the counts of the resulting trees on $6$ vertices are recorded, and we can see that only three of them remain.
    
    \input{Sections/Other/orbitreductiontable.tex}
    
    Let $M$ be the size of the orbit of a tree on $k$ vertices obtained by reduction. In the case where $k=2s$, we claim that the number of times a given orbit $\O$ appears is given by $\frac{s!M}{2^s}$. We can prove this by counting the number of trees on $n$ vertices to which this corresponds: we have $s!$ ways of placing complete binary trees on the $s$ white vertices and $M$ different ways to determine the underlying structure, which will count every corresponding tree on $n$ vertices $2^s$ times because each such tree has $2^s$ symmetries.
    
    Now we determine when this quantity is odd. We have
    \[\xi_2\left(\frac{s!M}{2^s}\right)=\xi_2(s!)+\xi_2(M)-\xi_2(2^s)=s-s_2(s)+\xi_2(M)-s=\xi_2(M)-s_2(s).\]
    This is odd if and only if $\xi_2(M)=s_2(s)$. Since $s_2(s)=s_2(2s+1)-1=s_2(k)-1$, this occurs if and only if we have a minimal orbit.
    
    Similarly, when $k=2s+1$, the number of times that a tree appears is given by $\frac{(s+1)!M}{2^s}$, which is odd only when we have a minimal orbit for the same reason.
    
    Up to now, we have shown that if $f(\mathcal{O})\geq 1$, i.e. an orbit of trees on $k$ vertices actually appears after the reduction operation, then $f(\mathcal{O})$ is even if $\mathcal{O}\not \in\mathcal{U}_k^{\text{min}}$. If $\mathcal{O}\in \mathcal{U}_k^{\text{min}}$, then $\mathcal{O}$ will be counted an odd number of times if we know how to assign the nodes as white/black to ``expand'' the tree back to the original size, though this still leaves out the possibility of multiple or no white/black assignments for $\O$. To finish the claim, we show that if $\mathcal{O}\in\mathcal{U}_k^{\text{min}}$, then there is a unique assignment of black and white nodes to $\O$ so that some orbits in $\mathcal{U}^{\text{min}}_n$ reduce to $\O$ with that color assignment. We know that nodes with no children after reduction must be white and those with at least one child must be black after reduction, which shows uniqueness. Using Theorem \ref{thm:orbit structure}, we may directly verify that trees in an orbit of $\mathcal{U}_k^{\text{min}}$ will always have a number of white nodes which is equal to or one more than the number of black nodes, so this coloring is always valid.
    
    With the claim, we now have
    \[\sum_{\mathcal{O}\in\mathcal{U}_n^{\text{min}}}\varepsilon^{\mathcal{O}}_0= \varepsilon_0^{n-k}\sum_{\mathcal{O}\in \mathcal{U}_k}f(\mathcal{O})\varepsilon^{\mathcal{O}}_0 \equiv\varepsilon^{n-k}_0\sum_{\mathcal{O}\in\mathcal{U}_k^{\text{min}}}\varepsilon^{\mathcal{O}}_0\pmod{2}\]
    so we may reduce to the case of $k$ vertices. We always have $k < n$ if $n\geq 2$ because $n+1\geq 3$, meaning that minimal orbits contain complete binary trees of depth greater than $1$ and the reduction operation always removes a positive number of vertices.
    
    For our base cases, we will eventually reduce to $n=1$ or $n=2$, where the statement is directly verified. Specifically, we have
    \begin{align*}
        \sum_{\mathcal{O}\in\mathcal{U}_1^{\text{min}}}\varepsilon^{\mathcal{O}}_0 &= \varepsilon_0, \\
        \sum_{\mathcal{O}\in\mathcal{U}_2^{\text{min}}}\varepsilon^{\mathcal{O}}_0 &= \varepsilon_0(\varepsilon_0+\varepsilon_1).
    \end{align*}
    The assumptions on $b$ are precisely that $\varepsilon_0=1$ and $\varepsilon_1=0$, so this is odd, as desired.
    
    Note that the reduction to a smaller case preserves parity, so we reduce to the $n=1$ or $n=2$ case according to the parity of $n$ at the start. Thus we actually get a stronger result: if we relax the condition $4\mid \Delta b$ to $2\mid \Delta b$, then $\xi_2(C^b_n)=\xi_2(C_n)$ holds for odd $n$, while $\xi_2(C^b_n) > \xi_2(C_n)$ for even $n$ if $4\nmid \Delta b$.
\end{proof}

An \cite{an2010combinatorial} previously determined necessary and sufficient conditions, conjectured by Konvalinka \cite{konvalinka2007divisibility}, under which a polynomial weight function $b(x)=b_0+b_1x+b_2x^2+\dotsb\in \mathbb{Z}[x]$ satisfies $\xi_2(C^b_n)=\xi_2(C_n)$. Inspired by his results, we make the following conjecture.
\begin{conjecture}
\label{conj:maingeneralization}
If $b: \Z_{\geq 0}\to\Z$ satisfies 
\begin{enumerate}
    \item $b(0)$ is odd,
    \item $2^{n-s_2(n)}\mid (\Delta^nb)(x)$ for all $n\geq 2$ and $x\in \mathbb{Z}_{\geq 0}$,
    \item $b(0)\equiv b(1)\pmod{4}$,
\end{enumerate}
then $\xi_2(C_n^b)=\xi_2(C_n)=s_2(n+1)-1.$
\end{conjecture}
We would be particularly interested in seeing a combinatorial proof of Conjecture \ref{conj:maingeneralization} as An's overall strategy was very computational in nature.
\section{Generalizations to prime powers}\label{sec:generalization}

Catalan numbers may be generalized to $q$-Catalan numbers, which have the formula 
\[C^{(q)}_n=\frac{1}{(q-1)n+1}\binom{qn}{n}.\]
In this section, we will use $C^{(q)}_n$ to denote $q$-Catalan numbers and $C^{(q)}_n(b)$ to denote weighted $q$-Catalan numbers.

These count the number of paths from $(0, 0)$ to $(qn, 0)$ with steps $(1, q-1)$ and $(1, -1)$ which never go below the $x$-axis and the number of $q$-ary trees on $n$ vertices (each vertex has $q$ distinguishable branches). Similarly, there are the weighted $q$-Catalan numbers. For a $q$-ary tree $\mathcal{T}$, weight each vertex $v$ by $b(i)$ where $i$ is the number of non-right edges on the path from the root to $v$, and let the weight of $\mathcal{T}$ be the product of the weights of its vertices. The weighted $q$-Catalan numbers are the sums of the weights of all $q$-ary trees on $n$ vertices:
\[C_n^{(q)}(b)=\sum_{\mathcal{T}}\wt(\mathcal{T}).\]
The weighted $q$-Catalan numbers have generating function
\[\sum_{n\geq0}C_n^{(q)}(b)x^n=\cfrac{1}{1-\cfrac{b(0)x}{\left(1-\cfrac{b(1)x}{1-\left(\cfrac{b(2)x}{(1-\cdots)^{q-1}}\right)^{q-1}}\right)^{q-1}}}.\]
As before, there is a symmetry group $G_n$ on $q$-ary trees of depth $n$, generated by permuting the subtrees at any vertex. The complete $q$-ary trees are fixed under all symmetries and consist of layers of $1, q, q^2, \dotsc, q^{k-1}$ vertices.

Konvalinka \cite{konvalinka2007divisibility} generalized Postnikov and Sagan's result to divisibility of $q$-Catalan numbers when $q$ is a power of a prime. Similarly, we may extend the techniques for the case $q=2$ to prove the following result:

\begin{theorem}\label{thm:qmain}
Let $q=p^k$ be a prime power. If $b: \Z_{\geq 0}\to\Z$ satisfies 
\begin{enumerate}
    \item $b(0)\equiv 1\pmod{q}$
    \item $q^2\mid (\Delta b)(x)$ for all $x\in\mathbb{Z}_{\geq 0}$,
    \item $q^n\mid (\Delta^nb)(x)$ for all $n\geq 2$ and $x\in \mathbb{Z}_{\geq 0}$,
\end{enumerate}
then
$C_n^{(q)}(b))\equiv C_n^{(q)}\pmod{p^{\xi+k}}$
where
\[\xi = \frac{s_p((q-1)n+1)-1}{p-1}=\xi_p(C_n^{(q)}).\]
\end{theorem}

The proof is similar so we will show the main claims and only give brief details of the proofs. Define $\mathcal{F}$ as the set of functions satisfying $q^n\mid \Delta^n f$ for all $n$, and the map $f\mapsto (\varepsilon_0^f, \varepsilon_1^f, \dotsc)\in (\mathbb{Z}/q\mathbb{Z})^{\mathbb{N}}$ so that $\Delta^n f\equiv \varepsilon^f_nq^n\pmod{q^{n+1}}$.

\begin{lemma}
The set $\mathcal{F}$ is closed under the following operations:
\begin{align*}
    &f\mapsto Sf, \quad (f, g)\mapsto f\cdot g, \quad \text{ and } \\
    &(f_1,\dotsc, f_q)\mapsto \langle f_1, \dotsc, f_q\rangle = \frac{1}{q}\sum_{i=1}^{q}f_i(x+1)\prod_{j\neq i}f_j(x).
\end{align*}
\end{lemma}
\begin{proof}
    Closure under $S$ is clear, and closure under products follows from Lemma \ref{lem:productrule} (the product rule). Closure under $\langle \dotsb\rangle$ follows by writing
    \[\langle f_1, \dotsc, f_q\rangle =\frac{1}{q}\sum_{i=1}^{q}(f_i+\Delta f_i)\prod_{j\neq i}f_j,\]
    expanding, and using the product rule.
\end{proof}
We define average weight functions $r_b(\mathcal{O}; x)$ for orbits as before, and these weight functions are in $\mathcal{F}$. Now we give the analog of Lemma \ref{lem:epsilon formula}.

\begin{lemma}
\label{lem:qepsilonformula}
Let $\mathcal{O}$ be an orbit of $q$-ary trees and let the orbits of the subtrees of the root vertex be $\O_1, \dotsc, \O_q$. Then
\[\varepsilon^{\mathcal{O}}_{m} \equiv\sum_{i_1+\dotsb+i_q+k=m}\binom{m}{i_1, \dotsc, i_q, k}\varepsilon_{k}(\varepsilon^{\mathcal{O}_1}_{i_1}\dotsm\varepsilon^{\mathcal{O}_q}_{i_q}+\varepsilon^{\mathcal{O}_1}_{i_1+1}\dotsm\varepsilon^{\mathcal{O}_q}_{i_q}+\dotsb+\varepsilon^{\mathcal{O}_1}_{i_1}\dotsm\varepsilon^{\mathcal{O}_q}_{i_q+1})\pmod{q}.\]
(The expression inside the parentheses contains $q+1$ terms, one term with indices $i_1, \dotsc, i_q$ while the other $q$ have exactly one index incremented.)
\end{lemma}
\begin{proof}
We use the recursive formula
\[r_b(\mathcal{O})=b\cdot \langle r_b(\O_1), \dotsc, r_b(\O_q)\rangle.\]
By expanding $\langle f_1, \dotsc, f_q\rangle =\frac{1}{q}\sum_{i=1}^{q}(f_i+\Delta f_i)\prod_{j\neq i}f_j$, we arrive at the given formula.
\end{proof}
Before stating the analog of Lemma \ref{lem:reduction}, we first have the following proposition:

\begin{proposition}
\label{prop:binomialdivisibility}
Let $q$ be a positive integer and $i_1, \dotsc, i_q$ be nonnegative integers, not all zero. Suppose that the distinct elements of the set $\{i_1, \dotsc, i_q\}$ are $j_1, \dotsc, j_t$ and these elements appear with multiplicity $k_1, \dotsc, k_t$. Then $q$ divides
\[\binom{i_1+\dotsc +i_q}{i_1, \dotsc, i_q}\binom{q}{k_1, \dotsc, k_t}.\]
\end{proposition}
\begin{proof}
    The binomial expression given counts the number of ways to assign $i_1+\dotsb+i_q$ distinguishable people into $q$ distinguishable rooms so that the distribution of the number of people in the rooms is a permutation of $i_1, \dotsc, i_q$. Since there is at least one person, we can group valid assignments with cyclic shifts (move each person down one room), so the number of assignments it a multiple of $q$.
\end{proof}

\begin{lemma}
\label{lem:qreduction}
    Let $\mathcal{O}$ be an orbit of $q$-trees on $n$ vertices, and suppose that for some vertex $v$, the subtree with root $v$ is a complete $q$-ary tree of depth $k$. Let $\mathcal{O}'$ be the same orbit of trees with the subtree at $v$ replaced by a single vertex at $v$. Then \[\varepsilon_m^{\mathcal{O}}=\varepsilon_0^{\frac{q^k-1}{q-1}-1}\varepsilon_m^{\mathcal{O}'}.\]
\end{lemma}
\begin{proof}
    Like the proof of Lemma \ref{lem:reduction}, it suffices to analyze the case where $\O$ is a complete $q$-ary tree and $v$ is its root. We expand with Lemma \ref{lem:qepsilonformula}. Since the orbits $\mathcal{O}_1, \dotsc, \mathcal{O}_q$ are the same, we have
    \[\sum_{i_1+\dotsb+i_q+k=m}\binom{m}{i_1, \dotsc, i_q, k}\varepsilon_{k}(\varepsilon^{\mathcal{O}_1}_{i_1+1}\dotsm\varepsilon^{\mathcal{O}_q}_{i_q}+\dotsb+\varepsilon^{\mathcal{O}_1}_{i_1}\dotsm\varepsilon^{\mathcal{O}_q}_{i_q+1})=0\pmod{q}.\]
    The remaining term is\[\varepsilon^{\mathcal{O}}_{m} \equiv\sum_{i_1+\dotsb+i_q+k=m}\binom{m}{i_1, \dotsc, i_q, k}\varepsilon_{k}\varepsilon^{\mathcal{O}_1}_{i_1}\dotsm\varepsilon^{\mathcal{O}_1}_{i_q}\pmod{q}.\]
    (We have replaced all $\O_i$ with $\O_1$.) Given a multiset of indices $\{i_1, \dotsc, i_q\}$, with distinct elements $j_1, \dotsc, j_t$ that appear $k_1, \dotsc, k_t$ times, the number of times the term $\varepsilon^{\mathcal{O}_1}_{i_1}\dotsm \varepsilon^{\O_1}_{i_q}$ appears in the above sum is $\binom{q}{k_1,\dotsc, k_t}$, and it is multiplied by $\binom{m}{i_1, \dotsc, i_q, k}$. Since $\binom{i_1+\dotsb +i_q}{i_1, \dotsc, i_q}$ divides $\binom{m}{i_1, \dotsc, i_q, k}$, by Proposition \ref{prop:binomialdivisibility}, these terms cancel out unless $i_1=\dotsb=i_q=0$. Hence we are left with
    \[\varepsilon^{\O}_m=\varepsilon_m(\varepsilon^{\O_1}_0)^q\]
    and the statement of the lemma follows from there.
\end{proof}

The proof of Theorem \ref{thm:qmain} is quite similar to the proof of Theorem \ref{thm:main} from here. We construct all minimal orbits on $n$ vertices according to the analog of Theorem \ref{thm:orbit structure}, and each one may be reduced using Lemma \ref{lem:qreduction} to get orbits of trees on a smaller number of vertices. This will reduce to the case of trees on a smaller number of vertices until we have reduced to $n\leq q$. In these cases, we directly verify that $\varepsilon_0=1$ and $\varepsilon_1=0$ is sufficient to have the sum of $\varepsilon^{\O}_0$ equal to $1\pmod{q}$.
\section{Periodicity of weighted Catalan numbers}\label{sec:period}
In this section we examine the periodicity of the weighted Catalan numbers modulo a positive integer $m$. In Section~\ref{subsec:period}, we prove the main result which determines if $\{C^b_n\pmod{m}\}$ is periodic. In Section~\ref{subsec:morseperiodicity}, we analyze the specific case of Morse link numbers previously studied by Postnikov and compute periods modulo $7, 11$, and $3^r$. Finally, in Section~\ref{subsec:valuationconjectures} we suggest some conjectures on Morse link numbers to be further explored.

\subsection{Determining when periodicity exists}\label{subsec:period}

We begin with the following lemma on finite continued fractions of the form in Proposition \ref{prop:genfunc}.

\begin{lemma}
\label{lem:cfracexpansion}
For integers $u, v$, let $S_k(u, v)$ denote the set of all sequences $u\leq i_1 < i_2 < \dotsc < i_k\leq v$ of integers such that $i_{m+1}-i_m\geq 2$ for all $m$. For a sequence $b(0), b(1), \dotsc, b(n)$, we have
\[\cfrac{1}{1-\cfrac{b(0)x}{1-\cfrac{b(1)x}{1-\cfrac{b(2)x}{1-\cfrac{\cdots}{1-b(n)x}}}}}=\frac{P(x)}{Q(x)}\]
where $P(x)$ and $Q(x)$ are defined as follows:
\begin{align*}
    P(x) &= 1+\sum_{k\geq 1}\left(\sum_{(i_1, \dotsc, i_k)\in S_k(1, n)}b(i_1)\dotsm b(i_k)\right)(-x)^k \\
    Q(x) &= 1+\sum_{k\geq 1}\left(\sum_{(i_1, \dotsc, i_k)\in S_k(0, n)}b(i_1)\dotsm b(i_k)\right)(-x)^k.
\end{align*}
\end{lemma}
\begin{proof}
    We induct on $n$. For $n=0$ the equality is $\frac{1}{1-b(0)x}=\frac{1}{1-b(0)x}$, which is clearly true. For the inductive step, assume the lemma is true for some $n$. Then
    \[\cfrac{1}{1-\cfrac{b(0)x}{1-\cfrac{b(1)x}{1-\cfrac{b(2)x}{1-\cfrac{\cdots}{1-b(n+1)x}}}}}=\frac{1}{1-b(0)x\frac{P_1(x)}{Q_1(x)}}=\frac{Q_1(x)}{Q_1(x)-b(0)xP_1(x)}\]
    where
    \begin{align*}
    P_1(x) &= 1+\sum_{k\geq 1}\left(\sum_{(i_1, \dotsc, i_k)\in S_k(2, n+1)}b(i_1)\dotsm b(i_k)\right)(-x)^k, \\
    Q_1(x) &= 1+\sum_{k\geq 1}\left(\sum_{(i_1, \dotsc, i_k)\in S_k(1, n+1)}b(i_1)\dotsm b(i_k)\right)(-x)^k
    \end{align*}
    by the inductive hypothesis. Now we have
    \[Q_1(x)-b(0)xP_1(x)=1+\sum_{k\geq 1}\left(\sum_{(i_1, \dotsc, i_k)\in S_k(0, n+1)}b(i_1)\dotsm b(i_k)\right)(-x)^k\]
    because the $P_1(x)$ term accounts for all sequences in $S_k(0, n+1)$ which have $i_1=0$ and the $Q_1(x)$ term accounts for all sequences in $S_k(0, n+1)$ which have $i_1\neq 0$. This proves the inductive step.
\end{proof}

We now have the following key theorem which describes when $C^b_n$ is eventually periodic modulo a positive integer $m$.
\begin{theorem}
\label{thm:periodicity}
Let $m$ be a positive integer. The sequence $\{C^b_n\pmod{m}\}$ is eventually periodic if and only if $m\mid b(0)\dotsm b(k)$ for some positive integer $k$.
\end{theorem}
\begin{proof}
    Using the Dyck path interpretation of Catalan numbers, suppose that $n\mid b(0)\dotsm b(k)$ for some positive integer $k$. Then all paths which exceed height $k$ have a weight that is $0\pmod{m}$, so they can be ignored. This means we may truncate the continued fraction for the weighted Catalan numbers at $b(k)$, so by Lemma \ref{lem:cfracexpansion}, the weighted Catalan numbers have a rational generating function modulo $m$, i.e.
    \[\sum_{n\geq 0}C^b_nx^n\equiv \frac{P(x)}{Q(x)}\pmod{m}\]
    for some polynomials $P, Q$. Furthermore, the constant term of $Q(x)$ is $1$ so this implies a linear recurrence relation $C^b_n\equiv a_1C^b_{n-1}+a_2C^b_{n-2}+\dotsb+a_kC^b_{n-k}\pmod{m}$ for sufficiently large $n$. This proves eventual periodicity because there are only finitely many possibilities for $(C^b_{n-1}, \dotsc, C^b_{n-k})\pmod{m}$, so the sequence eventually repeats.
    
    In the other direction, it suffices to only consider the case where $m=p^r$ for a prime number $p$ and positive integer $r$. We prove the claim by induction on $r$. Suppose that $C^b_n=C^b_{n+k}$ for some period $k$ and sufficiently large integers $n$. Then the generating function of $C^b_n$ is congruent modulo $p^r$ to a function of the form $\frac{P(x)}{Q(x)}$ for polynomials $P, Q$ with constant term $1$, as we may take $Q(x)=1-x^k$ and choose $P$ appropriately. (The choice of $P$ has constant term $1$ because $C^b_0=1$.)
    
    In the base case $r=1$, meaning $m=p$ is a prime, assume for the sake of contradiction that no $b(i)$ is a multiple of $p$. We have
    \[\cfrac{1}{1-\cfrac{b(0)x}{1-\cfrac{b(1)x}{1-\cfrac{b(2)x}{1-\cdots}}}}=\sum_{n\geq0}C_n^bx^n=\frac{P(x)}{Q(x)}=\frac{1+xP_1(x)}{1+xQ_1(x)}\pmod{p}\]
    for some polynomials $P_1, Q_1$ as $P, Q$ both have constant term $1$. By taking the reciprocal, subtracting both sides from $1$, and dividing by $b(0)x$ (recall that we are assuming that $b(0)$ is not a multiple of $p$), we get
    \[\cfrac{1}{1-\cfrac{b(1)x}{1-\cfrac{b(2)x}{1-\cfrac{b(3)x}{1-\cdots}}}}=b(0)^{-1}\frac{P_1(x)-Q_1(x)}{1+xP_1(x)}=\frac{b(0)^{-1}(P_1(x)-Q_1(x))}{P(x)}\pmod{p}.\]
    The constant term of $b(0)^{-1}(P_1(x)-Q_1(x))$ must be $1$ because the constant term of the left hand side is $1$, so we may repeat this procedure. Define $\{U_k\}, \{V_k\}$ by $R_0=P, S_0=Q$ and
    \begin{align*}
        U_{k+1} &= b(k)^{-1}\left(\frac{U_k-V_k}{x}\right) \\
        V_{k+1} &= U_k.
    \end{align*}
    Following the above argument, $U_k$ and $V_k$ are always polynomials with constant term $1$ and
    \[\cfrac{1}{1-\cfrac{b(k)x}{1-\cfrac{b(k+1)x}{1-\cfrac{b(k+2)x}{1-\cdots}}}}=\frac{U_k(x)}{V_k(x)}\pmod{p}.\]
    If $\deg U_k < \deg V_k$, then $\deg U_{k+1} = \deg V_k-1$ and $\deg V_{k+1} = \deg U_k < \deg V_k$, so $\max(\deg U_{k+1}, \deg V_{k+1}) < \max (\deg U_k, \deg V_k)$. Otherwise, if $\deg U_k\geq \deg V_k$, then $\deg U_{k+1}\leq \deg U_k-1 < \deg V_{k+1}$. Therefore the quantity $\max(\deg U_{k+1}, \deg V_{k+1})$ is nonincreasing and cannot remain constant twice in a row, so eventually we must have $U_k=0$. This is a contradiction as the continued fraction on the left has constant term $1$.
    
    Now for $r > 1$, we proceed similarly, starting with 
    \[\cfrac{1}{1-\cfrac{b(0)x}{1-\cfrac{b(1)x}{1-\cfrac{b(2)x}{1-\cdots}}}}=\sum_{n\geq0}C_n^bx^n=\frac{P(x)}{Q(x)}\pmod{p^r}\]
    for some polynomials $P, Q$ with constant term $1$. By the same argument as before, there must be some $b(i)$ which is a multiple of $p$. Suppose that $\nu_p(b(i))=\alpha$, and let $b(i)=p^{\alpha}\ell$ with $\gcd(\ell, p)=1$. Then
     \begin{equation}
         \cfrac{p^{\alpha}\ell}{1-\cfrac{b(i+1)x}{1-\cfrac{b(i+2)x}{1-\cfrac{b(i+3)x}{1-\cdots}}}}=\frac{U_i(x)-V_i(x)}{U_i(x)}\pmod{p^r}\label{eq:p^rcase}.
     \end{equation}
     We may multiply out to obtain
     \[p^{\alpha}\ell U_i(x)=(U_i(x)-V_i(x))(1-b(i+1)x+\dotsb)\pmod{p^r}\]
     where the term on the right is $1-\frac{b(i+1)x}{1-\frac{b(i+2)x}{1-\cdots}}$. Every coefficient on the left hand side is a multiple of $p^{\alpha}$, and we may argue inductively that every coefficient of $U_i-V_i$ must be a multiple of $p^{\alpha}$ as well, as the constant term of the other factor is $1$. Therefore we may divide (\ref{eq:p^rcase}) by $p^{\alpha}\ell$ to get
     \[\cfrac{1}{1-\cfrac{b(i+1)x}{1-\cfrac{b(i+2)x}{1-\cfrac{b(i+3)x}{1-\cdots}}}}=\frac{U_{i+1}(x)}{V_{i+1}(x)}\pmod{p^{r-\alpha}}\]
     for some new polynomials $U_{i+1}, V_{i+1}$. By the inductive hypothesis, we must have $p^{r-\alpha}\mid b(i+1)b(i+2)\dotsm b(k)$ for some $k$, so that $p^r\mid b(0)b(1)\dotsm b(k)$ as desired.
\end{proof}

\subsection{Periodicity of Morse link numbers} \label{subsec:morseperiodicity}
Morse curves and links were defined by Postnikov \cite{postnikov2000morse}, who showed that if $L_n$ is the number of combinatorial types of Morse links of order $n$, then $L_n=C^b_n$ for weight function $b(x)=(2x+1)^2$.

The proof of Theorem \ref{thm:periodicity} shows how to compute the period when it exists, as we may truncate the generating function to find a linear recurrence which can be solved with known tools. We will demonstrate (Example~\ref{ex:Lnmod7} and Example~\ref{ex:Lnmod11}) by solving some conjectures by Postnikov on the periods of $\{L_n\pmod{7}\}$ and $\{L_n\pmod{11}\}$.
\begin{example}\label{ex:Lnmod7}
    Since $b(3)=7^2$ is divisible by $7$, the generating function of $\{L_n\pmod{7}\}$ is given by
    \[\cfrac{1}{1-\cfrac{x}{1-\cfrac{3^2x}{1-5^2x}}}=\frac{1-34x}{1-35x+25x^2}\equiv \frac{1+x}{1+4x^2}\pmod{7}.\]
    This is a linear recurrence $L_0=L_1=1, L_{n+2}=-4L_n$ for $n\geq 0$. The explicit solution is 
    \[L_n=\left(\frac{1-r_2}{r_1-r_2}\right)r_1^n+\left(\frac{r_1-1}{r_1-r_2}\right)r_2^n\pmod{7}\]
    where $r_1, r_2\in \mathbb{F}_{49}$ are the roots of the characteristic equation $4x^2+1=0$. These are square roots of $5$, so $r_1^{12}\equiv 5^6\equiv 1\pmod{7}$ and they have order $12$. This is the minimal period of $L_n\pmod{7}$.
\end{example}

\begin{example}\label{ex:Lnmod11}
    Since $b(5)=11^2$ is divisible by $11$, the generating function of $\{L_n\pmod{11}\}$ is
    \[\cfrac{1}{1-\cfrac{x}{1-\cfrac{3^2x}{1-\cfrac{5^2x}{1-\cfrac{7^2x}{1-9^2x}}}}}\equiv \frac{1+x+5x^2}{1+6x^2+10x^3}\pmod{11}.\]
    The denominator factors as $(1-5x)(1-3x)^2$ and the solution is
    \[L_n= 6\cdot 3^n+10\cdot n3^n+6\cdot 5^n\pmod{11}.\]
    The orders of $3$ and $5$ modulo $11$ are both $5$, while $n\pmod{11}$ has period $11$. Hence $L_n\pmod{11}$ has period $55$ and it is not hard to verify that this is the minimal period.
\end{example}

In general, under the conditions of Theorem \ref{thm:periodicity}, we can only say that $\{C^b_n\pmod{m}\}$ is eventually periodic. However, under certain circumstances we can prove pure periodicity with Lemma \ref{lem:cfracexpansion} by using the following proposition:

\begin{proposition}
\label{prop:pureperiodicity}
Let $m$ be a positive integer and $P, Q$ be polynomials such that $\deg P < \deg Q$, and the constant and leading coefficients of $Q$ are coprime to $m$. Then the sequence $\{a_n\}$ with generating function $\frac{P}{Q}$ is purely periodic modulo $m$.
\end{proposition}
\begin{proof}
By reading off the coefficients of $Q$ we get a linear recurrence
\[a_{n+k}\equiv c_{1}a_{n+k-1}+\dotsb+c_ka_{n}\pmod{m}\]
which is valid for all $n$ because $\deg P < \deg Q$. As we have seen, this implies eventual periodicity. Furthermore, $c_k$ is coprime to $m$ so this linear recurrence can be extended backwards, so eventual periodicity implies pure periodicity in this case.
\end{proof}

\begin{corollary}
For any prime $p\equiv 3\pmod{4}$, the sequence $\{L_n\pmod{p}\}$ is purely periodic.
\end{corollary}
\begin{proof}
Let $k=\frac{p-3}{2}$. Then $\{L_n\pmod{p}\}$ has generating function
\[\cfrac{1}{1-\cfrac{b(0)x}{1-\cfrac{b(1)x}{1-\cfrac{b(2)x}{1-\cfrac{\cdots}{1-b(k)x}}}}}=\frac{P(x)}{Q(x)}\]
for the polynomials $P, Q$ given by Lemma \ref{lem:cfracexpansion}. The degree of $P$ is $\left\lceil\frac{k}{2}\right\rceil=\frac{p-3}{4}$ while the degree of $Q$ is $\left\lceil\frac{k+1}{2}\right\rceil = \frac{p+1}{4}$. Therefore $\deg P < \deg Q$. Also, the constant term of $Q$ is $1$ and the leading coefficient is $(-1)^{k/2}b(0)b(2)\dotsm b(k)$, which is not a multiple of $p$. Therefore we may apply Proposition \ref{prop:pureperiodicity}.
\end{proof}

Analyzing the period modulo prime powers is more difficult in general. The last result in this section will be to bound the period of $\{L_n\pmod{3^r}\}$, partially resolving this conjecture by Postnikov:

\begin{conjecture}[Postnikov \cite{postnikov2000morse}]
    Let $r\geq 3$ be a positive integer. The sequence $\{L_n\pmod{3^r}\}$ is purely periodic with period $2\cdot 3^{r-3}$.
\end{conjecture}

Our result is the following:
\begin{theorem}
\label{thm:Morse3period}
    Let $r\geq 3$ be a positive integer. The sequence $\{L_n\pmod{3^r}\}$ is eventually periodic with period dividing $2\cdot 3^{r-3}$.
\end{theorem}

The strategy of the proof is to classify Dyck paths of semilength $n$ based on edges with weight divisible by $3$. We show that within each class of paths, the sum of weights is eventually periodic in $n$ with period dividing $2\cdot 3^{r-3}$ (except for a small exception). Before proving Theorem \ref{thm:Morse3period}, we have two technical lemmas:

\begin{lemma}
\label{lem:binomperiod}
    Let $p$ be a prime and $m$ be a positive integer. The period of the sequence $\{\binom{n}{m}\pmod{p}\}$ divides the least prime power $p^k$ such that $p^k >m$.
\end{lemma}
\begin{proof}
    This follows from Lucas's theorem, because $\binom{n}{m}$ only depends on the digits of $n$ in base $p$ for which the corresponding digit in $m$ is nonzero.
\end{proof}

\begin{lemma}
\label{lem:primepowerperiod}
    Let $\{a_n\}$ be a sequence satisfying a linear recurrence
    \[a_n=c_1a_{n-1}+\dotsb+c_ka_{n-k}\]
    with initial conditions $a_1=a_2=\dotsb = a_{k-1}=0, a_{k}=1$. This sequence is eventually periodic modulo any positive integer $m$; let $\lambda(m)$ denote the period. Then for any prime $p$ and positive integer $r$, the period $\lambda(p^r)$ divides $p^{r-1}\lambda(p)$.
\end{lemma}
The proof is given in Theorem 3 of \cite{bright2008periodicity}.

\begin{proof}[Proof of Theorem \ref{thm:Morse3period}]
    For any Dyck path $P$ of semilength $n$, define its \emph{$3$-power path} $\alpha(P)$ as the Dyck path obtained by only considering all edges to or from $y=k$ where $k\equiv 1\pmod{3}$, and also marking all vertices for which the two adjacent edges are in opposite directions and were adjacent in $P$.
    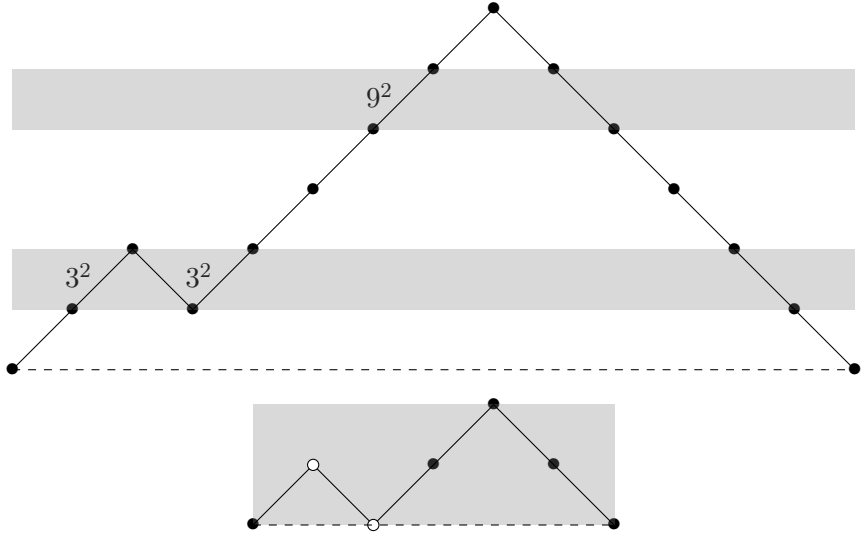
\begin{figure}[h!]
    \centering
    \begin{tikzpicture}[scale=0.8]
    \node at (0,0) {$\bullet$};
    \node at (1,1) {$\bullet$};
    \node at (2,2) {$\bullet$};
    \node at (3,1) {$\bullet$};
    \node at (4,2) {$\bullet$};
    \node at (5,3) {$\bullet$};
    \node at (6,4) {$\bullet$};
    \node at (7,5) {$\bullet$};
    \node at (8,6) {$\bullet$};
    \node at (9,5) {$\bullet$};
    \node at (10,4) {$\bullet$};
    \node at (11,3) {$\bullet$};
    \node at (12,2) {$\bullet$};
    \node at (13,1) {$\bullet$};
    \node at (14,0) {$\bullet$};
    \node[above left] at (1.5, 1.2) {$3^2$};
    \node[above left] at (3.5, 1.2) {$3^2$};
    \node[above left] at (6.5, 4.2) {$9^2$};
    \fill[fill=gray, fill opacity = 0.3] (0, 1)--(14, 1)--(14, 2)--(0, 2)--cycle;
    \fill[fill=gray, fill opacity = 0.3] (0, 4)--(14, 4)--(14, 5)--(0, 5)--cycle;
    \draw(0,0)--(2,2)--(3,1)--(8,6)--(14,0);
    \draw[dashed](0,0)--(14,0);
    \end{tikzpicture}
    
    \begin{tikzpicture}[scale=0.8]
    \node at (0,0) {$\bullet$};
    \node at (3,1) {$\bullet$};
    \node at (4,2) {$\bullet$};
    \node at (5,1) {$\bullet$};
    \node at (6,0) {$\bullet$};
    \fill[fill=gray, fill opacity = 0.3] (0, 0)--(6, 0)--(6, 2)--(0, 2)--cycle;
    \draw(0,0)--(1, 1)--(2,0)--(4,2)--(6,0);
    \node[white] at (1,1) {};
    \node[white] at (2,0) {};
    \draw[dashed](0,0)--(6,0);
    \end{tikzpicture}
    \caption{A Dyck path for $n=7$, and its $3$-power path. The white vertices are marked because $P$ stays within the gray region.}
    \label{fig:powerpath}
    \end{figure}
    
    See Figure \ref{fig:powerpath} for an example. Visually, we shade in all levels for which the weight is a multiple of three and reduce only to steps which fall within the shaded regions. The marked vertices (white vertices) indicate when the path stayed in a shaded region for two consecutive steps.
    
    If $\alpha(P)$ has semilength $k$, then the weight of $P$ is divisible by $3^{2k}$ because each upward edge of $\alpha(P)$ corresponds to an upward edge of $P$ with weight divisible by $3^2$. We may write
    \[L_n=\sum_{P}\wt(P)=\sum_{\beta}\sum_{P: \alpha(P)=\beta}\wt(P)\]
    where the first sum is over all Dyck paths $P$ of semilength $n$, the second sum is over all $3$-power paths $\beta$, and the third sum is over all $P$ of semilength $n$ such that $\alpha(P)=\beta$. Modulo $3^r$, we only need to consider $\beta$ of semilength at most $\left\lfloor\frac{r-1}{2}\right\rfloor$ because otherwise $\wt(P)$ is divisible by $3^r$.
    
    Next we develop a formula for $\sum_{P: \alpha(P)=\beta}\wt(P)$ when $\beta$ is fixed. Let $\beta$ have semilength $k$. Consider a vertex $v$ of $\beta$ and let $e_L$ and $e_R$ be the edges to the left and right. These two edges correspond to some edges $e_L'$ and $e_R'$ of the original path $P$. Let $G(v)$ be the edges between $e_L'$ and $e_R'$. (If $v$ is the first or last vertex of $\beta$, then $G(v)$ consists of the first or last segment of $P$ that does not touch a gray region.)
    
    Note that by definition, $G(v)$ is empty whenever $v$ is a marked vertex. Furthermore, $G(v)$ is always nonempty whenever $v$ is not marked, because if $e_1$ and $e_2$ are in the same direction then $e_L'$ and $e_R'$ cannot be adjacent in $P$ because they are at different levels, and if they are not in the same direction then $e_L'$ and $e_R'$ cannot be adjacent in $P$ because $v$ was not marked.
    
    Enumerating the vertices of $\beta$ as $v_0, v_1, \dotsc, v_{2k}$ and the edges as $e_1, e_2, \dotsc, e_{2k}$ the path $P$ may be written as the union \[P=G(v_0)\cup e_1 \cup G(v_1)\cup \dotsb \cup e_{2k}\cup G(v_{2k}).\]
    We have already observed that from $\beta$ alone, we can determine if $G(v)$ is empty or not. Let $V_{0}$ be the set of vertices for which $G(v)$ is empty, and $V_1$ be the set of vertices for which $G(v)$ is nonempty. Furthermore, the parity of $\lvert G(v)\rvert$ is also always determined: it is even except when $v$ is the first or last vertex because the start and endpoints of $G(v)$ are at the same height or distance $2$.
    
    To create all paths $P$ with $\alpha(P)=\beta$, we will choose the tuple $(\lvert G(v_0)\rvert, \dotsc, \lvert G(v_{2k}\rvert)$ first, and then for each $v$ determine what the edges of $G(v)$ are conditioned on $\lvert G(v)\rvert$. We know some elements of $(\lvert G(v_0)\rvert, \dotsc, \lvert G(v_{2k}\rvert)$ are zero, while the rest are $1+2m$ or $2+2m$ where $m$ can be any nonnegative integer. The weight of $P$ can be written as the product of the weights of $G(v_i)$ and the edges $e_i$ associated with $\beta$.
    
    Consider a vertex $v\in V_1$. First consider the case where $v$ is a start or end vertex and $G(v)$ has $1+2m$ edges. Clearly there is only one possibility for $G(v)$, which is to alternate between $y=0$ and $y=1$.
    
    Now suppose $v$ is not a start or end vertex and $G(v)$ has $2+2m$ edges. In the case where the edges $e_{L}, e_{R}$ to the left and right of $v$ go upward and downward respectively, $G(v)$ must go up from the top of a gray strip, at the line $y=3j+2$ for some $k$. The first edge of $G(v)$ must go from $y=3j+2$ to $y=3j+3$, while the last edge of $G(v)$ must go from $y=3j+3$ to $y=3j+2$. The $2m$ edges in between can be partitioned into pairs: the second and third edge either follow the path $3j+3\to 3j+2\to 3j+3$ or $3j+3\to 3j+4\to 3j+3$, the same is true for the fourth and fifth edges, etc. There are $2^m$ possibilities for $G(v)$, and the sum of the weights of $G(v)$ over all possibilities is given by $b(3j+2)^2(b(3j+2)+b(3j+3))^m$.
    
    Likewise, if $e_{L}$ and $e_{R}$ are downward and upward respectively, the sum of the weights of $G(v)$ over all possibilities is given by $b(3j+3)^2(b(3j+2)+b(3j+3))^m$ for some $j$. (There is the exception where $v$ is at $y=0$ in $\beta$, which will be similar to the case where $v$ is at the ends and the total weight will be $1$.) If $e_{L}$ and $e_{R}$ are in opposite directions, then the sum of the weights of $G(v)$ over all possibilities is given by $b(3j+2)b(3j+3)(b(3j+2)+b(3j+3))^m$ for some $j$.
    
    These results can be summarized as follows: for each $v\in V_1$, we may assign a nonnegative integer $m_v$ so that $G(v)$ has $1+2m_v$ or $2+2m_v$ edges depending on whether $v$ is at the ends. The total number of edges is $2k+2\lvert V_1\rvert-2 +2\sum_{v\in V_1} m_v=2n$ so we have the constraint $\sum m_v = n-k-\lvert V_1\rvert+1$. For each $v\in V_1$, we may independently determine $G(v)$, and the sum of the weights across the possible $G(v)$ is of the form $d_va_v^{m_v}$ for some constants $d_v, a_v$. A key property here is that $a_v$ is never a multiple of $3$, because it is either $1$ or $b(3j+2)+b(3j+3)\equiv 2\pmod{3}$. (Recall that our weight function is $b(x)=(2x+1)^2$.)   
    
    
    In summary, for a given path $3$-power path $\beta$, we have
    \[\sum_{P: \alpha(P)=\beta}\wt(P)=D_{\beta}\sum_{i_1+\dotsb+i_{\ell}=n-c_{\beta}}a_1^{i_1}\dotsm a_{\ell}^{i_\ell}\]
    where 
    \begin{align*}
        \ell &= \lvert V_1\rvert \\
        c_{\beta} &=k+\lvert V_1\rvert-1 \\
        D_{\beta} &= \prod_{e\in \beta}\wt(e)\prod_{v\in V_1}d_v
    \end{align*}
    are some constants depending only on $\beta$, and the $a_i$ are constants depending only on $\beta$ which are also not multiples of $3$.
    Denote 
    \[f(n)=\sum_{i_1+\dotsb+i_{\ell}=n-c_{\beta}}a_1^{i_1}\dotsm a_{\ell}^{i_\ell}.\]
    We have the generating function
    \[\sum_{n\geq 0}f(n)x^n=\sum_{n\geq 0}\sum_{i_1+\dotsb+i_{\ell}=n-c_{\beta}}a_1^{i_1}\dotsm a_{\ell}^{i_\ell}x^n=\frac{x^{c_{\beta}}}{(1-a_1x)\dotsm (1-a_{\ell}x)}.\]
    Modulo $3$, this takes the form $\frac{x^{c_\beta}}{(1-x)^{a}(1+x)^b}$ for some $a, b$ with $a+b=\ell$. It is known that this implies with this generating function may be written as
    \[f(n)\equiv \sum_{i=0}^{a-1}\binom{n}{i}+\sum_{i=0}^{b-1}\binom{n}{i}(-1)^n\]
    for $n\geq c_{\beta}$ (see \cite{pinch1995recurrent}). By Lemma \ref{lem:binomperiod}, the period of the above sequence modulo $3$ is a divisor of $2\cdot 3^m$ where $3^m$ is the least power of $3$ which is greater than $\max(a, b)$.
    
    For now we will assume that $\beta$ goes above the line $y=1$. This implies several things: first $\ell\geq 4$ because at least four vertices are in $V_1$: the start and end, and the two vertices where the path must cross the line $y=1$. It also implies that $D_{\beta}$ is a multiple of $3^{2k+2}$ because one of the edges has weight $9^2$.
    
    Note that $\ell\leq 2k+1$. Then we can say that $m\leq \ell-1\leq 2k$ as $3^{\ell-1} > \ell \geq \max(a, b)$ for $\ell\geq 4$. The period of the sequence modulo $3$ will divide $2\cdot 3^{2k}$. We also have the initial conditions $f(c_{\beta})=1$ and $f(n)=0$ for all $n < c_{\beta}$. By Lemma \ref{lem:primepowerperiod}, the period of $f(n)\pmod{3^{r-2k-2}}$ is a divisor of $2\cdot 3^{r-3}$. As $D_{\beta}$ is divisible by $3^{2k+2}$, the period of $\sum_{P: \alpha(P)=\beta}\wt(P)\pmod{3^{r}}$ is a divisor of $2\cdot 3^{r-3}$.
    
    Now we consider cases where $\beta$ does not go above $y=1$. Such paths must zigzag between $y=0$ and $y=1$. We also assume that $\beta$ is not the empty path, as that path has $\sum_{\beta}\sum_{P: \alpha(P)=\beta}\wt(P)=1$ which is periodic. The generating function takes the form
    \[\frac{x^{c_\beta}}{(1-x)^a(1-74x)^b}\]
    for some $a, b$ (here $74=5^2+7^2$). We also have the condition $a\geq 2$, because the start and end vertices of $\beta$ always have $a_v=1$. Let $3^m$ be the least power of $3$ greater than $\max(a, b)-1$. Then the period of this sequence modulo $3$ is divides $2\cdot 3^{m}$. If $m\leq \ell-3=a+b-3$, then the period of $f(n)\pmod{3^{r-2k}}$ divides $2\cdot 3^{r-3}$ which will be enough. We can directly check that this rules out all cases with $\max(a, b)\geq 5$, and further verification of the finitely many remaining cases leaves the following pairs:
    \[(a, b) = (2, 0), (3, 0), (4, 0),  (2, 1).\]
    In the case $(a, b)=(4, 0)$, the period of $f(n)\pmod{3^{r-2k}}$ divides $3^{r-2k+1}$. Therefore we only need to consider when the semilength $k$ is at most $1$. In fact, there is no path with $(a, b)=(4, 0)$ and semilength at most $1$, so we can eliminate this case.
    
    In the other three cases, the period of $f(n)\pmod{3^{r-2k}}$ divides $2\cdot 3^{r-2k}$ and again it suffices to only consider $k\leq 1$. This leaves the following two paths $\beta_1, \beta_2$:
    
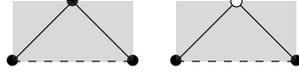
\begin{center}
    \begin{tikzpicture}[scale=0.8]
    \node at (0,0) {$\bullet$};
    \node at (1,1) {$\bullet$};
    \node at (2,0) {$\bullet$};
    \fill[fill=gray, fill opacity = 0.3] (0, 0)--(2, 0)--(2, 1)--(0, 1)--cycle;
    \draw(0,0)--(1, 1)--(2,0);
    \draw[dashed](0,0)--(2,0);
    \end{tikzpicture}
    \begin{tikzpicture}[scale=0.8]
    \node at (0,0) {$\bullet$};
    \node at (2,0) {$\bullet$};
    \fill[fill=gray, fill opacity = 0.3] (0, 0)--(2, 0)--(2, 1)--(0, 1)--cycle;
    \draw(0,0)--(1, 1)--(2,0);
    \draw[dashed](0,0)--(2,0);
    \node[white] at (1,1) {};
    \end{tikzpicture}
    \end{center}
    We will group these paths together. With all the factors included, we have
    \begin{align*}
        \sum_{P: \alpha(P)=\beta_1}\wt(P) &= 9\cdot 25\cdot\sum_{i_1+i_2+i_3=n-3}74^{i_2} = 9\cdot 25\cdot \sum_{k=0}^{n-3}(n-2-k)74^k \\
        &=9\cdot 25\cdot\left(\frac{74^{n-1}}{73^2}-\frac{n}{73}+\frac{72}{73^2}\right). \\
        \sum_{P: \alpha(P)=\beta_2}\wt(P) &= 9(n-1). \\
    \end{align*}
    The sum of these two is
    \[g(n)=9\left(\frac{25}{73^2}\cdot 74^{n-1}+\frac{48}{73}n-\frac{3529}{73^2}\right).\]
    We claim that $T=2\cdot 3^{r-3}$ is a period of this function of $n$ modulo $3^r$. It suffices to show that $2\cdot 3^{r-3}$ is a period of $\frac{25}{73^2}\cdot 74^{n-1}+\frac{48}{73}n-\frac{3529}{73^2}\pmod{3^{r-2}}$. However, this is easy because $74^n$ has period $2\cdot 3^{r-3}$ (since $\varphi(3^{r-2})=2\cdot 3^{r-3}$) and $\frac{48}{73}n$ has period $3^{r-3}$ because $48$ is divisible by $3$.
    
    We finish the argument as follows: any $\beta$ with semilength greater than $\left\lfloor\frac{r-1}{2}\right\rfloor$ may be ignored, so we only have finitely many $\beta$ to consider in the summation
    \[L_n=\sum_{\beta}\sum_{P: \alpha(P)=\beta}\wt(P).\]
    For each $\beta$, the sum $\sum_{P: \alpha(P)=\beta}\wt(P)$ is eventually periodic in $n$ with period dividing $2\cdot 3^{r-3}$ (aside from the two paths which we grouped together). As there are finitely many $\beta$, we conclude that $L_n\pmod{3^r}$ is eventually periodic with period dividing $2\cdot 3^{r-3}$.
\end{proof}

\subsection{Further Morse link number conjectures}\label{subsec:valuationconjectures}

Postnikov also stated some conjectures regarding $2$ and $5$-adic valuations of sequences related to $L_n$. We rephrase some below and provide some further conjectures.

\begin{conjecture}[\cite{postnikov2000morse}]
    There exists a $2$-adic integer $\alpha=\dotsc 010111_2$ such that
    \[\xi_2(L_n-C_n)=s_2(n)+\xi_2(n-\alpha)+2\]
    for all $n\geq 2$.
\end{conjecture}

We generalize this conjecture as follows:

\begin{conjecture}
    Let $k$ be a positive integer and $L^{(k)}_n$ denote the weighted Catalan numbers with weight $b(x)=(2x+1)^{2k}$. There exists a $2$-adic integer $\alpha_k$ and nonnegative integer $c_k$ such that
    \[\xi_2(L_n^{(k)}-C_n)=s_2(n)+\xi_2(n-\alpha_k)+c_k\]
    for all $n\geq 2$.
\end{conjecture}

Similar phenomena were \emph{not} observed for other polynomial weight functions $b$ which satisfy the conditions of Theorem \ref{thm:postsagan}, suggesting that there is something special here.

\begin{conjecture}[\cite{postnikov2000morse}]
\label{conj:Morse5adic}
    If $\alpha$ is the $5$-adic integer $\dotsc 111111120_5$, then for $n\geq 4$ we have
    \[\xi_5(L_n)=\begin{cases}
    2 &\text{ if $n$ is even} \\
    \xi_5(n-\alpha)+3 &\text{ if $n$ is odd}
    \end{cases}.\]
\end{conjecture}
Note that the original conjecture in \cite{postnikov2000morse} contains some typos.

From computational evidence, it appears that $\xi_3(L_n-1)$ also follows a type of pattern but it may not be as simple as the one given in Conjecture \ref{conj:Morse5adic}. We suggest the following:

\begin{conjecture}
\label{conj:Morse3adic}
    There exists a $3$-adic integer $\alpha$ such that for odd $n\geq 3$, the quantity $\xi_3(L_n-1)$ depends only on $\xi_3(n-\alpha)$ and the last digit of $\frac{n-\alpha}{3^{\xi_3(n-\alpha)}}$, which is $1$ or $2$.
\end{conjecture}
If Conjecture \ref{conj:Morse3adic} were true, then for $n\geq 3$ we would have
\[\xi_3(n) =\begin{cases}
2 & n\text{ even} \\
6 & n\equiv 1\pmod{6} \\
4 & n\equiv 3\pmod{6} \\
5 & n\equiv 5, 11 \pmod{18} \\
\vdots
\end{cases}.\]

\section*{Acknowledgements}
This research was carried out as part of the 2019 Summer Program in Undergraduate Research (SPUR) of the MIT Mathematics Department. We would like to thank Prof.\,Alex Postnikov for suggesting the project and Prof.\,Richard Stanley, Prof.\,Ankur Moitra and Prof.\,David Jerison for helpful conversations.

\bibliographystyle{plain}
\bibliography{ref}

\begin{thebibliography}{10}

\bibitem{an2010combinatorial}
Junkyu An.
\newblock {\em Combinatorial enumeration of weighted Catalan numbers}.
\newblock PhD thesis, Massachusetts Institute of Technology, 2010.

\bibitem{bright2008periodicity}
Curtis Bright.
\newblock Modular periodicity of linear recurrence sequences.
\newblock {\em preprint}, 2008.
\newblock available at
  \url{https://cs.uwaterloo.ca/~cbright/reports/PM434Project.pdf}.

\bibitem{deutsch2006congruences}
Emeric Deutsch and Bruce~E. Sagan.
\newblock Congruences for {C}atalan and {M}otzkin numbers and related
  sequences.
\newblock {\em J. Number Theory}, 117(1):191--215, 2006.

\bibitem{dickson1966history}
Leonard~Eugene Dickson.
\newblock {\em History of the theory of numbers. {V}ol. {I}: {D}ivisibility and
  primality}.
\newblock Chelsea Publishing Co., New York, 1966.

\bibitem{goulden2004combinatorial}
Ian~P. Goulden and David~M. Jackson.
\newblock {\em Combinatorial enumeration}.
\newblock Dover Publications, Inc., Mineola, NY, 2004.
\newblock With a foreword by Gian-Carlo Rota, Reprint of the 1983 original.

\bibitem{konvalinka2007divisibility}
Matja\v{z} Konvalinka.
\newblock Divisibility of generalized {C}atalan numbers.
\newblock {\em J. Combin. Theory Ser. A}, 114(6):1089--1100, 2007.

\bibitem{pinch1995recurrent}
R.~G.~E. Pinch.
\newblock Recurrent sequences modulo prime powers.
\newblock In {\em Cryptography and coding, {III} ({C}irencester, 1991)},
  volume~45 of {\em Inst. Math. Appl. Conf. Ser. New Ser.}, pages 297--310.
  Oxford Univ. Press, New York, 1993.

\bibitem{postnikov2000morse}
Alexander Postnikov.
\newblock Counting morse curves and links.
\newblock {\em preprint}, 2010.
\newblock available at
  \url{https://math.mit.edu/~apost/papers/morse-brief.pdf}.

\bibitem{postnikov2007power}
Alexander Postnikov and Bruce~E. Sagan.
\newblock What power of two divides a weighted {C}atalan number?
\newblock {\em J. Combin. Theory Ser. A}, 114(5):970--977, 2007.

\bibitem{stanley2015catalan}
Richard~P. Stanley.
\newblock {\em Catalan numbers}.
\newblock Cambridge University Press, New York, 2015.

\end{thebibliography}

\end{document}